\documentclass{article}
\usepackage{authblk}
\usepackage[utf8]{inputenc}
\usepackage{amsmath,amsfonts,amsthm}
\usepackage{enumitem}
\usepackage{color}
\usepackage{graphicx}

\newcommand{\ccurl}{\operatorname{curl}}
\newcommand{\ddiv}{\operatorname{div}}
\newcommand{\cA}{\mathcal{A}}
\def\bR{\mathbb{R}}
\def\cT{\mathcal{T}}
\def\cE{\mathcal{E}}
\def\cV{\mathcal{V}}
\def\bx{\boldsymbol{x}}
\newcommand{\NA}{\mathbb{N}}
\newcommand{\norm}[1]{\left\Vert #1 \right\Vert}

\newtheorem{lemma}{Lemma}
\newtheorem{remark}{Remark}
\newtheorem{theorem}{Theorem}
\newtheorem{proposition}{Proposition}
\newenvironment{proof*}[1]
  {\renewcommand{\proofname}{#1}%
   \begin{proof}}
  {\end{proof}}
\title{An adaptive finite element scheme for the Hellinger--Reissner elasticity mixed eigenvalue problem}
\author[1,2]{Fleurianne Bertrand}
\author[2,3]{Daniele Boffi}
\author[1,4]{Rui Ma}
\affil[1]{Humboldt-Universit\"at zu Berlin, Germany}
\affil[2]{King Abdullah University of Science and Technology, Saudi Arabia}
\affil[3]{University of Pavia, Italy}
\affil[4]{Universit\"at Duisburg-Essen, Germany}
\date{}

\begin{document}

\maketitle

\begin{abstract}
In this paper we study the approximation of eigenvalues arising from the mixed Hellinger--Reissner elasticity problem by using a simple finite element introduced recently in~\cite{HuMa2019}. We prove that the method converge when a residual type error estimator is considered and that the estimator decays optimally with respect to the number of degrees of freedom. The analogue of a postprocessing technique introduced in~\cite{GedickeKhan2018} is discussed and tested numerically.
\end{abstract}
\section{Introduction}

In this paper we introduce, discuss, and analyze a finite element scheme for the approximation of the eigenmodes associated with linear elasticity.

We consider a simple elasticity finite element introduced recently in~\cite{HuMa2019}, which is a modification of the Hu--Zhang element~\cite{HuZhang2015}.
In the framework of the mixed Hellinger--Reissner elasticity problem, the Hu--Zhang element approximates the stresses with continuous symmetric polynomials of degree $k$ enriched by $H(\ddiv)$ symmetric bubbles of degree $k$. The displacements are approximated by discontinuous polynomials of degree $k-1$ in each component.

A key property for the analysis of adaptive schemes is the nestedness of the finite element spaces: i.e., the standard analysis uses the fact that if $\cT_h$ is a refinement of $\cT_H$, then the finite element spaces defined on $\cT_H$ are contained in those built on $\cT_h$. Since this property is not satisfied by the original Hu--Zhang element, in~\cite{HuMa2019} a modification of the element has been introduced in order to guarantee the nestedness of the spaces after refinements based on newest vertex bisection.

The aim of our analysis is to extend the results of~\cite{HuMa2019} to the corresponding eigenvalue problem. By combining the abstract theory of~\cite{CarstensenFPD2014} and~\cite{CarstensenH2017} with the results of~\cite{BoffiGGG} we can prove that a residual type error estimator $\eta$ can drive an adaptive scheme in such a way that $\eta$ decays optimally in terms of the number of degrees of freedom. Moreover, we discuss the application of a postprocessing technique introduced in~\cite{GedickeKhan2018} in order to improve the rate of convergence of the scheme.

\section{Elasticity eigenvalue problem and its discretization}
\subsection{Continuous problem}
Given a bounded Lipschitz polygonal domain $\Omega\subset\mathbb{R}^2$, we
define $\Sigma:=\{\tau\in H(\ddiv,\Omega;\mathbb{S})$, where $\mathbb{S}$ denotes the space of symmetric $2\times2$ matrices, and $V:=L^2(\Omega;\mathbb{R}^2)$.

The mixed formulation of the elasticity eigenvalue problem seeks $(\lambda,\sigma,u )\in\bR\times\Sigma\times V$ with $\norm{u}_{L^2}=1$ that solve
\begin{equation}
\aligned
 (A\sigma,\tau)+(\ddiv\tau,u)&= 0 &&\text{for all } \tau\in\Sigma ,\\
   (\ddiv\sigma,v)&=- (\lambda u,v)  &&\text{for all } v\in V.\label{eq:conEigen}
\endaligned
\end{equation}

We assume that the compliance tensor $A$ is in $L^\infty(\mathbb{S},\mathbb{S})$ and positive definite uniformly in $\Omega$. Thanks to the regularity assumption on the domain the problem is associated with a compact solution operator. It is then well known that the eigenvalues can be numbered in an increasing order as follows:
\begin{align*}
    0<\lambda_1\leq\lambda_2\leq\lambda_3\leq\cdots;
\end{align*}
we denote the corresponding eigenfunctions by $\{(\sigma_1,u_1),(\sigma_2,u_2),\dots\}$.
\subsection{Discretization}
We are going to use the extended stress space on adaptive meshes introduced in~\cite{HuMa2019}. The finite element space is constructed locally on the actual elements without the use of a reference configuration and depends on the refinements; more precisely, the stresses are augmented by suitable shape functions in a neighborhood of the new vertices introduced by the adaptive algorithm.

Given an initial shape-regular triangulation $\cT_0$ of $\Omega$ into triangles, let $\cT_h$ denote  a refinement of $\cT_0$ after a finite number of successive bisections of triangles by the newest vertex bisection strategy. Given $K\in\cT_h$, let $P_k(K; X)$ denote as the space of
polynomials of degree $\leq k$, taking value  in the finite-dimensional vector space $X$. The range space $X$ will be either $\mathbb{S}$, $\mathbb{R}^2$ or $\mathbb{R}$.

Let us fix the polynomial degree $k\ge3$.
We recall the discrete stress space $\widetilde{\Sigma}_h$ and the displacement space $V_h$ from \cite{Hu2015,HuZhang2015}
\begin{align}
\begin{aligned}
\widetilde{\Sigma}_h:=&\{\tau\in \Sigma : \tau=\tau_c+\tau_b,\ \tau_c\in H^1(\Omega;\mathbb{S}),\ \tau_c|_K\in P_k(K;\mathbb{S}),
\\
& \qquad\tau_b|_K\in \Sigma_{\partial K,k,b}\text{ for all } K\in \mathcal{T}_h\},
\\
V_{h}:=&\{v\in L^2(\Omega;\bR^2) : v|_K\in P_{k-1}(K;\bR^2)\text{ for all }K\in \cT_h\},
\end{aligned}
\end{align}
where the full $H(\ddiv,K;\mathbb{S})$ bubble function space consisting of polynomials of degree $\leq k$ is
\begin{equation*}
\Sigma_{\partial K,k,b}:=\{\tau\in P_k(K;\mathbb{S}),\tau \nu|_{\partial K}=0\} 
\end{equation*}
with the outward unit normal $\nu$ of $\partial K$.

\begin{remark}
Functions in $\widetilde{\Sigma}_h$ are continuous at the vertices of $\cT_h$ and only $H(\ddiv)$ conforming along the edges. For this reason the stress spaces are not nested after mesh refinements. This is the main motivation for the introduction of a suitable modification in~\cite{HuMa2019}, where the continuity is relaxed at the vertices added by the adaptive refinement.
\end{remark}

Let $\cV_0$ denote the set of all vertices of $\cT_0$. Let $\cV_h$ (resp. $\cV_{h,0}$) denote the set of all (resp. internal) vertices of $\cT_h$. The newest-vertex bisection creates each new vertex $\bx\in \cV_{h,0}\setminus\cV_{0}$ as a midpoint of an old edge $e$ associated with a tangential vector $t_{\bx}$ and normal vector $\nu_{\bx}$ of $e$; the set of elements meeting at $\bx$ is split into two patches $\omega_{\bx}^+$ and $\omega_{\bx}^-$ by the edge $e$, namely
\begin{align*}
\omega_{\bx}^+&=\bigcup\{K\ |\ K\in\cT_h,\bx\in K, ({\rm mid}(K)-\bx)\cdot\nu_{\bx}>0\},\\
\omega_{\bx}^-&=\bigcup\{K\ |\ K\in\cT_h,\bx\in K, ({\rm mid}(K)-\bx)\cdot\nu_{\bx}<0\},
\end{align*}
where ${\rm mid}(K)$ is the barycenter of $K$.

Let $\phi_{\bx}$ denote the nodal basis at $\bx$ of the Lagrange element of order $k$ and define the space
\[
    E_h:=\operatorname*{span}_{\bx\in \cV_{h,0}\setminus\cV_{0}}\{\phi_{\bx}|_{\omega_{\bx}^+}t_{\bx}t_{\bx}^T,\phi_{\bx}|_{\omega_{\bx}^-}t_{\bx}t_{\bx}^T\},
\]
where the span is taken over all newly added vertices $\bx$ along the internal edges.
Then the extended stress space $\Sigma_h$ in \cite[Section~ 3.1]{HuMa2019} is defined as
\begin{align*}
    \Sigma_h:=\widetilde{\Sigma}_h+E_h.
\end{align*}

It is proved in~\cite[Theorem~3.2]{HuMa2019} that if $\cT_h$ is a refinement of $\cT_H$ then $\Sigma_H\subset\Sigma_h$.

The discretization of \eqref{eq:conEigen} seeks $(\lambda_h,\sigma_h,u_h)\in\bR\times\Sigma_h\times V_h$ with $\norm{u_h}_{L^2}=1$ such that
\begin{align}
\begin{aligned}
 (A\sigma_h,\tau_h)+(\ddiv\tau_h,u_h)&= 0 &  \text{ for all   } \tau_h\in\Sigma_h ,\\
   (\ddiv\sigma_h,v_h)&=- (\lambda_h u_h,v_h)  & \text{ for all \ } v_h\in V_h.\label{eq:disEigen}
\end{aligned}
\end{align}
The discrete eigenvalues can be enumerated as 
\begin{align*}
0<\lambda_{h,1}\leq\lambda_{h,2}\leq\cdots\leq\lambda_{h,N(h)}
\end{align*}
with corresponding eigenfunctions $\{(\sigma_{h,1},u_{h,1}),\cdots,(\sigma_{h,N(h)},u_{h,N(h)})\}$, where $N(h)={\rm dim}(V_h)$.

The a priori convergence of the discrete eigenmodes to the continuous one is a consequence of the error estimates proved in~\cite{HuMa2019,HuZhang2015}. Indeed, from Theorem~3.1 and Remark~3.2 of~\cite{HuZhang2015} the following theorem can be proved.

\begin{theorem}
Let $\lambda=\lambda_i=\dots\lambda_{i+m-1}$ be an eigenvalue of~\eqref{eq:conEigen} of multiplicity $m$ and let us denote by $\mathcal{G}$ the $m$-dimensional eigenspace spanned by $\{u_i,\dots,u_{i+m-1}\}$ and by $\mathcal{F}$ the space spanned by $\{\sigma_i,\dots\sigma_{i+m-1}\}$. Let $\mathcal{G}_h$ denote the space spanned by the corresponding discrete eigenfunctions $\{u_{h,i},\dots,u_{h,i+m-1}\}$ and let us define
\[
\varepsilon_h=\sup_{\substack{(u,\sigma)\in\mathcal{G}\times\mathcal{F}\\\|(u,\sigma)\|_{L ^2\times H(\ddiv)}=1}}\min_{(v_h,\tau_h)\in V_h\times\Sigma_h}(\|u-v_h\|_{L^2}+\|\sigma-\tau_h\|_{H(\ddiv)}).
\]
Then
\[
\aligned
&|\lambda-\lambda_{h,j}|\le C\varepsilon_h^2\quad&&j=i,\dots i+m-1\\
&\delta(\mathcal{G},\mathcal{G}_h)\le C\varepsilon_h,
\endaligned
\]
where $\delta(A,B)$ denotes the gap between the subspaces $A$ and $B$.
\end{theorem}

In the rest of this paper, $\cT_h$ will denote an arbitrary refinement of a fixed mesh $\cT_H$, while $\cT_\ell$ refers to the sequence designed by that adaptive procedure. The eigenmode approximation $(\lambda,\sigma,u)$ will be indicated by $(\lambda_\kappa,\sigma_\kappa,u_\kappa)$ where $\kappa$ may be $H$, $h$ or $\ell$, respectively.

\subsection{Error estimator and adaptive method}
\label{sec:resestimator}

In order to keep the notation as simple as possible, we consider the approximation of an eigenvalue $\lambda$ of~\eqref{eq:conEigen} of multiplicity equal to one. A corresponding eigenfunction is denoted by $(\sigma,u)$ and the eigenspace by $E$. More general situations need appropriate modifications in the spirit of~\cite{BoffiGGG,Gallistl2015}.

Let $\cE_h$ (resp. $\cE_h(\Omega)$
and $\cE_h(\partial\Omega)$) denote the collection of all (resp. interior and boundary)
element edges of $\cT_h$. For any triangle $K\in\cT_h$, let $\cE(K)$ denote the set of its edges and let $h_K:=|K|^{1/2} $ with $h :=\max_{K\in\cT}h_K$. For any edge $e\in \cE_h$, let $h_e:=|e|$ and let $t_e$
denote the unit tangential vector and let $\nu_e:=t_e^\perp$
denote the unit normal vector.   The jump $[w]_e$ of $w$ across edge $e = K_1\cap K_2$ reads $
[w]_e := (w|_{K_1}
)|_e- (w|_{K_2}
)|_e$.
Particularly, if $e\in\mathcal{E}_h(\partial\Omega)$, $[w]_e := w|_e$.
The scheme is based on the local error estimator
\begin{align}\label{eq:estimatorall}
\begin{aligned}
\eta^2(\cT_h,K): ={}& h_K^4\|\ccurl\ccurl(A\sigma_h)\|_{L^2(K)}^2\\
&+\sum_{e\in\cE(K)}\Big(h_{K}\|\mathcal{J}_{e,1}\|_{L^2(e)}^2+h_K^3\|\mathcal{J}_{e,2}\|_{L^2(e)}^2\Big)\\
&+h_{K}^2\norm{ A\sigma_h-\epsilon(u_h)}^2_{L^2(K)}+ \sum_{e\in\cE(K)}h_K\norm{[u_h]_e}_{L^2(e)}^2
\end{aligned}
\end{align}
with
$$
\begin{array}{ll}
 \mathcal{J}_{e,1}:=&\left\{\begin{array}{ll}
 \Big[(A\sigma_h)t_e\cdot t_e\Big]_e&{\rm if~} e\in\mathcal{E}_h(\Omega),\\
  \Big((A\sigma_h)t_e\cdot t_e \Big)\big|_e\quad\quad\quad\quad\quad\ & {\rm if~} e\in\mathcal{E}_h(\partial\Omega),
 \end{array}\right. \\
 \\
 \mathcal{J}_{e,2}:=&\left\{\begin{array}{lr}
 \Big[{\rm curl}(A\sigma_h)\cdot t_e\Big]_e\quad&\quad\hspace{1mm} {\rm if~} e\in\cE_h(\Omega),\\
 \Big({\rm curl}(A\sigma_h)\cdot t_e -\partial _{t_e}\big((A\sigma_H)t_e\cdot \nu_e\big)\Big)\big|_e& \hspace{-1mm} {\rm if~} e\in\mathcal{E}_h(\partial\Omega).\hspace{0.5mm}
 \end{array}\right.
 \end{array}
$$

As usual we define the global error estimator $\eta^2(\cT_h):=\sum_{K\in\cT_h}\eta^2(\cT_h,T)$ and, given any $\mathcal{M}\subseteq\cT_h$, we define $\eta^2(\cT_h,\mathcal{M}):=\sum_{K\in\mathcal{M}}\eta^2(\cT_h,K)$.

The adaptive scheme adopts then the standard SOLVE--ESTIMATE--MARK--REFINE strategy with D\"orfler marking corresponding to a bulk parameter $\theta\in(0,1)$ (where $\theta=1$ means uniform refinement and $\theta=0$ means no refinement). The usual structure of the algorithm (after $\ell\in\NA$ levels of refinement) is as follows.

\begin{description}[leftmargin=1.9cm,style=nextline]
\item[Solve:]
compute the discrete solution $(\lambda_\ell,\sigma_\ell,u_\ell)$ on the mesh $\cT_\ell$.
\item[Estimate:]
compute the local contributions of the error estimator $\eta^2(\cT_\ell)$.
\item[Mark:]
choose a minimal subset $\mathcal{M}_\ell\subset\cT_\ell$ such that $\theta\eta^2(\cT_\ell)\le\eta^2(\mathcal{M}_\ell)$
\item[Refine:]
generate a new triangulation $\cT_{\ell+1}$ (by the newest vertex bisection algorithm) as the smallest refinement of $\cT_\ell$ satisfying $\mathcal{M}_\ell\cap\cT_{\ell+1}=\emptyset$.
\end{description}

The convergence of the adaptive scheme is usually analyzed in the framework of nonlinear approximation classes. Given an initial triangulation $\cT_0$, the best convergence rate $s\in(0,+\infty)$ for the approximation of a space $W$ is characterized in terms of
\[
|W|_{\cA_s}=\sup_{N\in\NA}N^s\inf_{\cT\in\cT(N)}\delta(W,W_\cT),
\]
where $\cT(N)$ is the set of triangulations obtained by $\cT_0$ after adding at most $N$ elements, $W_\cT$ is an approximation of $W$ on the mesh $\cT$, and $\delta(A,B)$ is the gap between $A$ and $B$. In particular, $|W|_{\cA_s}<+\infty$ if $\delta(W,W_\cT)=O(N^{-s})$ for an optimal triangulation in $\cT(N)$.

The next theorem contains the main result of this paper: it states that the error estimators goes to zero optimally with respect to the number of degrees of freedom.

\begin{theorem}\label{thm:optimalconvergence}
Let $s\in(0,+\infty)$ be such that $\sup_{N\in\mathbb{N}_0}(1+N)^s\min\eta(\cT(N))<+\infty$
Provided that the meshsize of the initial mesh $\cT_0$ and the bulk parameter $\theta$ are small enough, the output of the adaptive algorithm satisfies
\begin{align*}
    \sup_{\ell\in\mathbb{N}_0}(1+|\cT_\ell|-|\cT_0|)^s\eta(\cT_\ell)\approx\sup_{N\in\mathbb{N}_0}(1+N)^s\min\eta(\cT(N)).
\end{align*}
\end{theorem}

In the statement of Theorem~\ref{thm:optimalconvergence} we follow the approach of~\cite{CarstensenFPD2014}, where sufficient conditions (axioms) for the convergence of the error estimator are considered.
If the eigenspace $E$ corresponding to $\lambda$ satisfies $|E|_{\cA_s}<+\infty$, then the optimal convergence of the actual error could then be obtained as a second step by using the efficiency of the error estimator. The interested reader is referred to~\cite[Section~4.1]{CarstensenFPD2014} for more links between these results.

The proof of Theorem~\ref{thm:optimalconvergence} is based on the following three propositions, see \cite{CarstensenFPD2014,CarstensenH2017}, which we are going to prove in the next section. In order to simplify the notation, from now on we omit the subscript for the $L^2$ norm which will be denoted simply by $\|\cdot\|$ and we denote by $\|\cdot\|_A$ the norm associated with the scalar product induced by the compliance matrix $(A\cdot,\cdot)$. Analogously, we continue to use the notation $(\cdot,\cdot)$ to indicate the $L^2$ scalar product in the whole domain $\Omega$.
\begin{proposition}
[discrete-reliability]\label{thm:discrete reliability} There exists a constant $C_{\rm drel}$ and a function $\rho_{\rm drel}(H)$ tending to zero, such that, for a sufficiently fine mesh $\cT_H$ and for all its refinements $\cT_h$, it holds that
\begin{align*}
\norm{\sigma_h-\sigma_H}_A+\norm{u_h-u_H}\leq{} &C_{\rm drel}\eta(\cT_H,\cT_H\setminus\cT_h)\\
&+\rho_{\rm drel}(H)(\norm{\sigma_h-\sigma_H}_A+\norm{u_h-u_H}).
\end{align*}
\end{proposition}
\begin{proposition}
[quasi-orthogonality] \label{thm:quasi-ortho}There exists a function $\rho_{qo}(h)$ tending to zero as $h$ goes to zero, such that
\begin{align*}
   \norm{\sigma_h-\sigma_H}_A^2&+ \norm{u_h-u_H}^2\leq \norm{\sigma-\sigma_H}_A^2+\norm{u-u_H}^2-\norm{\sigma-\sigma_h}_A^2-\norm{u-u_h}^2\\
   &+\rho_{qo}(h)(\norm{\sigma-\sigma_h}_A^2+\norm{u-u_h}^2+\norm{\sigma-\sigma_H}_A^2+\norm{u-u_H}^2).
\end{align*}
\end{proposition}
\begin{proposition}[contraction]\label{thm:contraction}
If the initial mesh $\cT_0$ is sufficiently fine, then there exist constants $\beta>0$ and $0<\gamma<1$ such that the term
\begin{align*}
     \xi_\ell^2:=\eta^2(\cT_\ell)+\beta(\norm{\sigma_\ell-\sigma_{\ell+1}}_A^2+\norm{u_\ell-u_{\ell+1}}^2)
\end{align*}
satisfies for all integers $\ell$
\begin{align*}
     \xi^2_{\ell+1}\leq\gamma\xi^2_{\ell}.
\end{align*}
\end{proposition}

\subsection{Proofs}
This section starts with some error estimates for the approximation of Problem~\eqref{eq:disEigen}. The first one is a superconvergence estimate.
\begin{lemma}\label{lem:supcon}Let $(\lambda,\sigma,u)$ solve \eqref{eq:conEigen} and let $(\lambda_h,\sigma_h,u_h)$ solve \eqref{eq:disEigen}. 
There exists a function $\rho_{sc}(h)$ tending to zero as $h\rightarrow 0$ such that
\begin{align}
\label{eq:supcon}
\norm{P_h u-u_h}\leq \rho_{sc}(h)(\norm{\sigma-\sigma_h}_A+\norm{u-u_h}) 
\end{align}
where $P_h$ denotes the $L^2$ projection onto $V_h$.
\end{lemma}
\begin{proof}
This result is standard in the framework of mixed problems~\cite{DouglasRoberts1982} and it has been extended to eigenvalue problems in~\cite[Proposition~6.5]{BoffiGGG} for the standard mixed finite elements for the mixed Poisson problem and in~\cite[Thm~3.5]{GedickeKhan2018} for Arnold--Winther mixed finite elements.
\end{proof}
The following relations between the error in the eigenvalues and the eigenfunctions is a common tool for the analysis of mixed schemes (see \cite[Lemma~4]{DuranGP1999})
\begin{align}\label{eq:relation}
\begin{aligned}
\lambda-\lambda_h&=\norm{\sigma-\sigma_h}_ A^2-\lambda_h\norm{u-u_h}^2,\\
\lambda_h-\lambda_H&=\norm{\sigma_h-\sigma_H}_A^2-\lambda_H\norm{u_h-u_H}^2.
\end{aligned}
\end{align}
From the a priori error estimates for $\norm{\sigma-\sigma_h}_A$ and $\norm{u-u_h}$ we then obtain the existence of $\rho_0(H)$ tending to zero as $H$ goes to zero such that
\begin{align}\label{eq:eslam}
|\lambda_h-\lambda_H|\leq\rho_0(H)(\norm{\sigma_h-\sigma_H}_A+\norm{u_h-u_H}).
\end{align}
The following auxiliary problem will be useful for the next results: Find $(\widehat{\sigma}_H,\widehat{u}_H)\in\Sigma_H\times V_H$ such that
\begin{align}
\begin{aligned}
 (A\widehat{\sigma}_H,\tau_H)+(\ddiv\tau_H,\widehat{u}_H)&= 0 &  \text{ for all   } \tau_H\in\Sigma_H ,\\
   (\ddiv\widehat{\sigma}_H,v_H)&=- (\lambda_h u_h,v_H)  & \text{ for all \ } v_H\in V_H.\label{eq:auxdisEigen}
\end{aligned}
\end{align}

The next two lemmas are the analogous of Lemma~4.2 and Lemma~4.3 of~\cite{BoffiG2019} and can obtained with obvious modification of the corresponding proofs.

 \begin{lemma}\label{lem:diffaux}
 Let $(\sigma_H,u_H)\in\Sigma_H\times V_H$ solve \eqref{eq:disEigen} on the mesh $\cT_H$ and $(\widehat{\sigma}_H,\widehat{u}_H)\in\Sigma_H\times V_H$ solve \eqref{eq:auxdisEigen} on the same mesh.  There exists $\rho_1(H)$ tending to zero as $H$ goes to zero such that
 \begin{align}
 \label{eq:diffaux}
 \norm{\sigma_H-\widehat{\sigma}_H}_{A}+\norm{u_H-\widehat{u}_H}\leq C\norm{\widehat{u}_H-P_H u_h}+\rho_1(H)(\norm{\sigma_h-\sigma_H}_A+\norm{u_h-u_H}).
 \end{align}
 \end{lemma}

 \begin{lemma}\label{lem:diffauxtwo}
 Let $(\sigma_h,u_h)\in\Sigma_h\times V_h$ solve \eqref{eq:disEigen} and $(\widehat{\sigma}_H,\widehat{u}_H)\in\Sigma_H\times V_H$ solve \eqref{eq:auxdisEigen}. Then for $H$ small enough, there exists $\epsilon>0$ only depending on $\Omega$ such that
 \begin{align*}
\norm{\widehat{u}_H-P_H u_h}\leq CH^\epsilon(\norm{\sigma_h-\sigma_H}_A+\norm{u_h-u_H}).
 \end{align*}
 \end{lemma}

\begin{proof}[Proof of Proposition~\ref{thm:discrete reliability}]
We first estimate the term $\norm{\sigma_h-\sigma_H}_A$. Introduce the auxiliary solution $(\zeta_h,r_h)\in\Sigma_h\times V_h$ of
 \begin{align}
 \label{eq:auxdis}
 \begin{aligned}
  (A\zeta_h,\tau_h)+(\ddiv\tau_h,r_h)&= 0 &&  \text{ for all   } \tau_h\in\Sigma_h ,\\
   (\ddiv\zeta_h,v_h)&=  (\ddiv(\sigma_h-\widehat{\sigma}_H),v_h)  && \text{ for all \ } v_h\in V_h,
 \end{aligned}
 \end{align}
where $\widehat{\sigma}_H$ was defined in~\eqref{eq:auxdisEigen}. Since $\ddiv(\sigma_h-\widehat{\sigma}_H-\zeta_h)=0$, the discrete Helmholtz decomposition \cite[Lemma~3.6]{HuMa2019} shows that there exists $\phi_h\in H^2(\Omega)$ in  the extended Argyris finite element space \cite{CarstensenHu2019} such that
\begin{align*}
 \sigma_h-\sigma_H=\ccurl\ccurl \phi_h+\zeta_h+\widehat{\sigma}_H-\sigma_H.
\end{align*}
Lemma~\ref{lem:diffaux} and \ref{lem:diffauxtwo} show that there exists $\rho_0(H)$ tending to zeros as $H$ goes to zero such that
\begin{align}\label{eq:termone}
\norm{\widehat{\sigma}_H-\sigma_H}_A\leq C\rho_0(H)(\norm{\sigma_h-\sigma_H}_A+\norm{u_h-u_H}).
\end{align}
Since $(A\zeta_h,\ccurl\ccurl\phi_h)=0$, it holds
\begin{align*}
 \norm{\ccurl\ccurl \phi_h}_A\leq C(\norm{\sigma_h-\sigma_H}_A+\rho_0(H)\norm{u_h-u_H}).
\end{align*}
The choice $\tau_h=\ccurl\ccurl\phi_h$ in \eqref{eq:disEigen} shows
\begin{align*}
 \norm{\ccurl\ccurl \phi_h}_A^2&=(A(\sigma_h-\widehat{\sigma}_H),\ccurl\ccurl \phi_h)=-(A\widehat{\sigma}_H ,\ccurl\ccurl\phi_h)\\
 &=-(A(\widehat{\sigma}_H -\sigma_H),\ccurl\ccurl\phi_h)-(A\sigma_H,\ccurl\ccurl\phi_h).
\end{align*}
It has been proven in \cite{HuMa2019} that
\begin{align}
-(A\sigma_H,\ccurl\ccurl\phi_h)\leq C \eta_1(\cT_H,\cT_H\setminus\cT_h)  \norm{\ccurl\ccurl \phi_h}_A,
\end{align}
where $\eta ^2_1(\cT_H,\cT_H\setminus\cT_h)=\sum_{K\in\cT_H\setminus\cT_h}\eta_1^2(\cT_H,K)$ is the error estimator from~\cite{ChenHHM2018}, that is,
\begin{align*} 
\eta_1^2(\cT_H,K): = h_K^4\|\ccurl\ccurl(A\sigma_H)\|_{L^2(K)}^2+\sum_{e\in\cE(K)}\Big(h_{K}\|\mathcal{J}_{e,1}\|_{L^2(e)}^2+h_K^3\|\mathcal{J}_{e,2}\|_{L^2(e)}^2\Big).
\end{align*}

It remains to estimate $\norm{\zeta_h}_A$. Since $\ddiv\zeta_h=\lambda_hu_h-P_H(\lambda_hu_h)$, the $L^2$ best approximation \cite[Thm~5.1]{CarstensenGS2016} shows 
\begin{align*}
\norm{\zeta_h}_A\leq CH\norm{\lambda_h u_h-P_H(\lambda_h u_h)}\leq CH\norm{u_h-u_H}.
\end{align*}
The combination of the previous estimates lead to
\begin{align}\label{eq:errsigma}
\norm{\sigma_h-\sigma_H}_A\leq C\big(\eta_1(\cT_H,\cT_H\setminus\cT_h)+(\rho_0(H)+H)(\norm{\sigma_h-\sigma_H}_A+\norm{u_h-u_H})\big).
\end{align}

The second step is to analyze $\norm{u_h-u_H}$. Adding and subtracting $P_H u_h$ and $\widehat{u}_H$ we have
\begin{align}\label{eq:sumU}
\begin{aligned}
u_h-u_H=(P_H u_h-\widehat{u}_H)+(\widehat{u}_H-u_H)-(P_H u_h-u_h).
\end{aligned}
\end{align}
Lemma~\ref{lem:diffaux} and \ref{lem:diffauxtwo} show that
\begin{align}\label{eq:sumUa}
\norm{P_H u_h-\widehat{u}_H} +\norm{\widehat{u}_H-u_H}\leq C\rho_0(H)(\norm{\sigma_h-\sigma_H}_A+\norm{u_h-u_H}).
\end{align}
It remains to estimate $P_Hu_h-u_h$. Given any $K\in\cT_H\setminus\cT_h$, since $P_Hu_h-u_h$ is zero mean valued on $K$, there exists $\tau\in H^1_0(K,\mathbb{S})$ such that $\ddiv\tau=P_Hu_h-u_h$ and $\norm{\tau}_{H^1(K)}\leq C\norm{P_Hu_h-u_h}$. Since $(P_Hu_h,\ddiv\tau)=(P_Hu_h,P_Hu_h-u_h)=0$ and $P_H\ddiv\tau=P_H(P_Hu_h-u_h)=0$, we have
\[
\aligned
(P_Hu_h-u_h,P_Hu_h-u_h)_{L^2(K)}&=(P_Hu_h-u_h,\ddiv\tau)_{L^2(K)}\\
&=-(u_h,\ddiv\tau)_{L^2(K)}=-(u_h,P_h\ddiv\tau)_{L^2(K)}\\
&=-(u_h,P_h\ddiv\tau-P_H\ddiv\tau)_{L^2(K)}.
\endaligned
\]
We can extend $\tau$ by zero to $\Omega$, so that we have
\begin{align*}
(P_Hu_h-u_h,P_Hu_h-u_h)_{L^2(K)}=(u_h,P_H\ddiv\tau-P_h\ddiv\tau).
\end{align*}
Let $\Pi_h$ (resp. $\Pi_H$) denote the interpolation presented in \cite[Remark~3.1]{Hu2015} with respect to the mesh $\cT_h$ (resp. $\cT_H$), which satisfies $P_h\ddiv\tau=\ddiv \Pi_h\tau$ (resp. $P_H\ddiv\tau=\ddiv \Pi_H\tau$).

We will also use later on that $\Pi_h\tau|_K=\Pi_H\tau|_K$ for any $K\in\cT_H\cap\cT_h$.

Then, using the properties of the interpolation, the nestedness of the spaces, and the variational formulation, we have
\[
\aligned
(P_Hu_h-u_h,P_Hu_h-u_h)_{L^2(K)}&=(u_h,P_H\ddiv\tau-P_h\ddiv\tau)\\
&=(u_h,\ddiv (\Pi_H-\Pi_h)\tau)\\
&=(A\sigma_h,(\Pi_h-\Pi_H)\tau)
\endaligned
\]
Adding and subtracting suitable quantities, we are led to the sum of the following three terms
\begin{align}\label{eq:sumUb}
\begin{aligned}
     \norm{P_Hu_h-u_h}^2_{L^2(K)}={}&(A(\sigma_h-\sigma_H),(\Pi_h-\Pi_H)\tau)\\
&+(A\sigma_H-\epsilon_H(u_H),(\Pi_h-\Pi_H)\tau)\\
&+(\epsilon_H(u_H),(\Pi_h-\Pi_H)\tau),   
\end{aligned}
\end{align}
where $\epsilon_H$ denotes the piecewise symmetric gradient on the mesh $\cT_H$.

The first term is bounded by $CH\norm{\sigma_h-\sigma_H}_A\norm{P_Hu_h-u_h}$. The second term is bounded by 
\begin{align*}
&(A\sigma_H-\epsilon_H(u_H),(\Pi_h-\Pi_H)\tau)\\
\leq &C \Big(\sum_{K\in\cT_H\setminus\cT_h}h_K^2\norm{A\sigma_H-\epsilon(u_H)}_{L^2(K)}^2\Big)^{1/2}\norm{P_Hu_h-u_h}.
\end{align*}
For the last term,  an integration by parts shows
\begin{align*}
 (\epsilon_H(u_H),(\Pi_h-\Pi_H)\tau)&=\sum_{e \in \cE_H}\int_e(\Pi_h-\Pi_H)\tau \nu_e\cdot[u_H]_e\,ds\\
 &\leq C\big(\sum_{K\in\cT_H\setminus\cT_h}\sum_{e\in\cE(K)}h_K\norm{[u_H]_e}_{L^2(e)}^2\big)^{1/2}\norm{P_Hu_h-u_h}.
\end{align*}
The previous   estimates applied in \eqref{eq:sumUb} and \eqref{eq:sumU}--\eqref{eq:sumUa} lead to
 \begin{align}
 \label{eq:estvect}
 \begin{aligned}
\norm {u_h-u_H}\leq &C\bigg(\Big(\sum_{K\in\cT_H\setminus\cT_h}\big(h_K^2\norm{A\sigma_H-\epsilon(u_H)}_{L^2(K)}^2+\sum_{e\in\cE(K)}h_K\norm{[u_H]_e}_{L^2(e)}^2\big)\Big)^{1/2}\\
&+(\rho_0(H)+H)(\norm{\sigma_h-\sigma_H}_A+\norm{u_h-u_H})\bigg).
 \end{aligned}
 \end{align}
A combination of \eqref{eq:errsigma} and \eqref{eq:estvect} concludes the proof.

\end{proof}


\begin{proof}[Proof of Proposition~\ref{thm:quasi-ortho}]

We have by direct computation that
\begin{align*}
\norm{\sigma_h-\sigma_H}^2_A&=\norm{\sigma-\sigma_H}_A^2-\norm{\sigma-\sigma_h}_A^2-2(A(\sigma-\sigma_h),\sigma_h-\sigma_H),\\
\norm{u_h-u_H}^2&=\norm{u-u_H}^2-\norm{u-u_h}^2-2(P_h u-u_h,u_h-u_H).
\end{align*}
The combination of \eqref{eq:conEigen} and \eqref{eq:disEigen} implies
\begin{align*}
(A(\sigma-\sigma_h),\sigma_h-\sigma_H)&=-(\ddiv(\sigma_h-\sigma_H),u-u_h)\\
&=(\lambda_hu_h-\lambda_Hu_H,u-u_h)\\
&=(\lambda_hu_h-\lambda_Hu_H,P_hu-u_h).
\end{align*}
Lemma~\ref{lem:supcon} and \eqref{eq:relation} show
\begin{align*}
&(A(\sigma-\sigma_h),\sigma_h-\sigma_H)+(P_h u-u_h,u_h-u_H)\\
&\quad=(\lambda_hu_h-\lambda_Hu_H,P_hu-u_h)+(u_h-u_H,P_h u-u_h)\\
&\quad\leq(|\lambda_h-\lambda_H|+(1+\lambda_H)\norm{u_h-u_H})\norm{P_hu-u_h}\\
&\quad\leq(\norm{\sigma_h-\sigma_H}_A^2+\lambda_H\norm{u_h-u_H}^2\\
&\qquad+(1+\lambda_H)\norm{u_h-u_H})\cdot\rho_{sc}(h)(\norm{\sigma-\sigma_h}_A+\norm{u-u_h}).
\end{align*}
This concludes the proof.

\end{proof}


\begin{proof}[Proof of Proposition~\ref{thm:contraction}]

The contraction property is quite standard in the framework of adaptive schemes.  It is a consequence of the following error estimator reduction property:  there exist constants $\beta_1\in(0,\infty)$  and $\gamma_1\in (0,1)$  such that, if $\cT_{\ell+1}$ is the refinement of $\cT_\ell$ generated by the adaptive scheme, then it holds that
\begin{align*}
    \eta^2(\cT_{\ell+1})\leq\gamma_1\eta^2(\cT_\ell)+\beta_1(\norm{\sigma_\ell-\sigma_{\ell+1}}_A^2+\norm{u_\ell-u_{\ell+1}}^2).
\end{align*}

\end{proof}

\section{Local reconstruction}
\label{sec:postestimator}

In this section we show how to apply the procedure presented in~\cite{GedickeKhan2018} to our approximate solution. Since the analysis of~\cite{GedickeKhan2018} can be carried over directly with the natural modifications to our situation, we limit ourselves to the description of what can be done. The aim is to construct locally an eigenfunction $u^*_h$ that can be used both for the definition of a postprocessed eigenvalue $\lambda^*_h$ and for the design of a new error estimator.

For $m\geq k+1$ we define
\begin{align*}
    V_h^\ast=\{v\in L^2(\Omega)\ |\ v|_K\in P_m(K;\bR^2)\text{ for all }K\in\cT_h\}.
\end{align*}
Given a pair of functions $(\sigma_h,u_h)\in\Sigma_h\times U_h$, on each $K\in\cT_h$ we set $P_K=P_h|_K$ and we define $u_h^\ast\in V_h^\ast$ as a solution of the system
\begin{align}
    \begin{aligned}
    P_ku_h^\ast&=u_h,\\
    (\epsilon(u_h^\ast),\epsilon(v_h))_{L^2(K)}&=(A\sigma_h,\epsilon(v))_{L^2(K)}\quad\text{for all }v\in (I-P_K)V_h^\ast|_K.
    \end{aligned}
\end{align}
This postprocessing has  been used in \cite{ChenHH2018,ChenHHM2018} for the source problem.

If $(u_h^\ast,u_h^\ast)\neq 0$ we can define the postprocessed eigenvalue as the value of the Rayleigh quotient of the postprocessed eigenfunction
\begin{align*}
    \lambda_h^\ast=-\frac{(\ddiv\sigma_h,u_h^\ast)}{(u_h^\ast,u_h^\ast)}.
\end{align*}

It is then possible in a natural way to define an error indicator based on the postprocessed solution as follows
\begin{align*}
    \eta_\ast^2:=\norm{A\sigma_h-\epsilon_h(u_h^\ast)}^2+\sum_{K\in\cT_h}h_K^2\norm{\lambda_h^\ast u_h^\ast+\ddiv\sigma_h}^2_{L^2(K)}+\sum_{e\in\cE_h}h_e^{-1}\norm{[u_h^\ast]_e}_{L^2(e)}^2.
\end{align*}


The following result expresses a reliability estimate and can be obtained as in~\cite[Thm~5.3]{GedickeKhan2018}

\begin{proposition}

The postprocessed eigenfunction $u_h^\ast$ satisfies the following reliability estimate:
\begin{align*}
    \norm{\sigma-\sigma_h}_A+\norm{\epsilon_h(u-u_h^\ast)}\leq C(\eta_*+\mathcal{Y})
\end{align*}
where $\mathcal{Y}$ is the higher order term
\begin{align*}
    \mathcal{Y}:=\Big(\sum_{K\in\cT_h}h_K^2\norm{\lambda u-\lambda_h^\ast u_h^\ast}^2_{L^2(K)}\Big)^{1/2}+\lambda_h\norm{u-u_h^\ast}+|\lambda-\lambda_h|.
\end{align*}

\end{proposition}

\section{Numerical results}


We conclude this paper with some numerical results of $k=3$.

Consider the L-shaped domain with $\Omega=(-1,1)^2\setminus[0 ,1]^2$. Let $A\sigma=\frac{1}{2\nu}(\sigma-\frac{1}{2}({\rm tr}\sigma)\delta)$ with $\nu=1$, which corresponds with a Stokes problem.
We are interested in the approximation of the first (singular) eigenvalue; we computed a reference solution with generalized Taylor--Hood element of order five on a uniform refinement of the last mesh that we obtained with our adaptive scheme. The obtained value is $\lambda=32.13269464746$.

In Figure~\ref{LshapeAdap} we report on the same plot all the obtained results.
We compare the error $|\lambda-\lambda_\ell|$ and the error of the postprocessed solution $|\lambda-\lambda_\ell^\ast|$ by using both the residual estimator $\eta$ from Section~\ref{sec:resestimator} and the estimator based on the postprocessed solution $\eta_{\ast}$ from Section~\ref{sec:postestimator} in the adaptive algorithm. we can see that the error $|\lambda-\lambda_\ell|$ is smaller by using $\eta_{\ast}$, while  the postprocessed error $|\lambda-\lambda_\ell^\ast|$ is almost the same.
It can be appreciated that in general the error corresponding to the postprocessed solution is smaller that the original one.
\begin{figure}\label{LshapeAdap}\centering
\includegraphics[width=10cm]{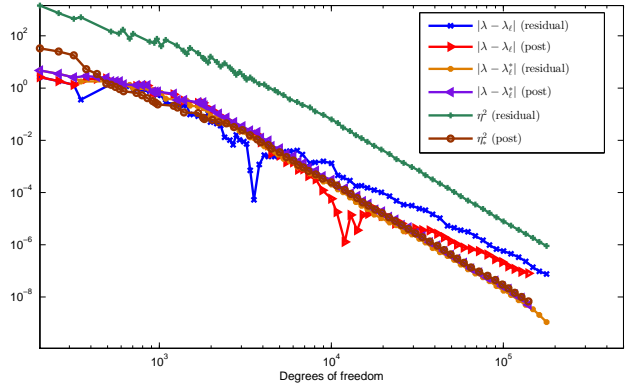}
\caption{Convergence history of the adaptive scheme for different choices of error estimator and error quantity}
\end{figure}

\end{document}